\documentclass[final,reqno,10pt]{amsart}
\usepackage{amsmath,amsthm,amssymb}
\usepackage{color,soul}
\usepackage{tikz}
\usepackage{hyperref}
\usepackage{mathrsfs}

\newcommand{\R}{\mathbb{R}}
\newcommand{\C}{\mathbb{C}}
\newcommand{\N}{\mathbb{N}}
\newcommand{\Z}{\mathbb{Z}}
\newcommand{\HH}{\mathcal{H}}
\newcommand{\D}{\mathcal{D}}

\newcommand{\We}{W_e}
\newcommand{\sige}{\sigma_e}

\newcommand{\tr}{\operatorname{tr}}
\newcommand{\re}{\operatorname{Re}}
\newcommand{\im}{\operatorname{Im}}
\newcommand{\conv}{\operatorname{conv}}

\newcommand{\inner}[1]{\left\langle #1 \right\rangle}

\newcommand{\cl}[1]{\overline{#1}}

\newcommand{\mf}[1]{\mathfrak{#1}}

\theoremstyle{plain}
\newtheorem{theorem}{Theorem}[section]

\newtheorem{corollary}[theorem]{Corollary}
\newtheorem{lemma}[theorem]{Lemma}
\newtheorem{proposition}[theorem]{Proposition}

\numberwithin{equation}{section}

\theoremstyle{definition}

\newtheorem{example}[theorem]{Example}

\theoremstyle{remark}
\newtheorem{remark}[theorem]{Remark}

\begin{document}
\title[The essential numerical range and a theorem of Simon]{The essential numerical range and a theorem of Simon on the absorption of eigenvalues}
\author[B. Lins]{Brian Lins}
\date{}
\address{Brian Lins, Hampden-Sydney College}
\email{blins@hsc.edu}


\begin{abstract}
Let $A(t)$ be a holomorphic family of self-adjoint operators of type (B) on a complex Hilbert space $\HH$. Kato-Rellich perturbation theory says that isolated eigenvalues of $A(t)$ will be analytic functions of $t$ as long as they remain below the minimum of the essential spectrum of $A(t)$. At a threshold value $t_0$ where one of these eigenvalue functions hits the essential spectrum, the corresponding point in the essential spectrum might or might not be an eigenvalue of $A(t_0)$. Our results generalize a theorem of Simon to give a sufficient condition for the minimum of the essential spectrum to be an eigenvalue of $A(t_0)$ based on the rate at which eigenvalues approach the essential spectrum. We also show that the rates at which the eigenvalues of $A(t)$ can approach the essential spectrum from below correspond to eigenvalues of a bounded self-adjoint operator. The key insight behind these results is the essential numerical range which was recently extended to unbounded operators by B\"{o}gli, Marletta, and Tretter. 
\end{abstract}

\maketitle

\section{Introduction}
Let $A(t)$ be a holomorphic family of self-adjoint operators on a complex Hilbert space $\HH$ with inner product $\inner{\cdot,\cdot}$. For bounded operators, this means that 
$$A(t) = A_0 + t A_1 + t^2 A_2 + \ldots$$
where each $A_k$ is a bounded self-adjoint operator on $\HH$, and there is an $r > 0$ such that the series converges absolutely in norm when $|t| < r$. If $\lambda_0$ is an isolated eigenvalue of $A(t_0)$ with finite multiplicity $m$, Kato-Rellich perturbation theory predicts that there is a holomorphic family of $m$ mutually orthogonal unit vectors $x_j(t) \in \HH$ such that each $x_j(t)$ is an eigenvector of $A(t)$ for $1 \le j \le m$ and $t$ in a neighborhood of $t_0$, and the corresponding eigenvalues $\lambda_j(t)$ satisfy $\lambda_j(t_0) = \lambda_0$ for each $1 \le j \le m$.  

This theory was first introduced by Rellich \cite{Rellich37} and later refined by Kato \cite{Kato} and Sz.-Nagy \cite{SzNagy46}. In addition to bounded self-adjoint families, the theory also applies to certain special families of unbounded self-adjoint operators, the so called type (A) and the more general type (B) holomorphic families \cite{Kato}. For both type (A) and type (B) families, the operators $A(t)$ are self-adjoint and bounded below. The main difference between the two types is that the domain of $A(t)$ is constant for type (A) families, but this is not assumed for type (B). Kato-Rellich theory applies to isolated eigenvalues with finite multiplicity in both cases, however.  

If $\lambda(t)$ is an element in the discrete spectrum of $A(t)$ that depends analytically on $t$ when $t > t_0$, but $\lambda(t)$ approaches an element of the essential spectrum of $A(t_0)$ as $t \rightarrow t_0^+$, then it is not possible in general to analytically continue $\lambda(t)$ beyond $t_0$, and in fact, $\lim_{t \rightarrow t_0^+} \lambda(t)$ may not even be an eigenvalue of $A(t_0)$.  Still, there are situations where it might be important to know the rate at which $\lambda(t)$ approaches the essential spectrum or to determine whether or not the limit of $\lambda(t)$ is an eigenvalue of $A(t_0)$. Such problems arise in mathematical physics when one considers perturbations of a Schrodinger operator $-\Delta + t V$ \cite{GesztesyHolden87,KlausSimon80,Rauch80} (see also \cite{Golovaty19} and the references therein).  They also arise when considering the boundary curves of the numerical range of a bounded operator \cite{LiSp20}.

One result, due to Simon \cite[Theorem 2.1]{Simon77} is the following.

\begin{theorem} \label{thm:Simon}
Let $A, B$ be self-adjoint operators on $\HH$ with $\D(|A|^{1/2}) \subseteq \D(|B|^{1/2})$. Suppose that $A \ge 0$, $0$ is in the essential spectrum of $A$, and $B$ is relative $A$-form compact, that is, $|B|^{1/2}(A+I)^{-1}|B|^{1/2}$ is compact. Suppose also that $A \dotplus tB$ has a largest negative eigenvalue $\mu(t)$ which is non-degenerate (i.e., has multiplicity one) for all $0 < t_0 < t < \epsilon$ for some $t_0$ and $\epsilon >0$. If $\mu(t)$ converges to $0$ as $t \rightarrow t_0^+$, and no other eigenvalue of $A \dotplus tB$ converges to $0$, then either 
\begin{enumerate}
\item $\lim_{t \rightarrow t_0^+} \mu(t)/(t-t_0) = 0$, or 
\item $0$ is an eigenvalue of $A \dotplus t_0B$. 
\end{enumerate} 
In the later case, suppose that $0$ is not an eigenvalue of $A$. Then $0$ is a simple eigenvalue of $A \dotplus t_0 B$ and 
$$\lim_{t \rightarrow t_0^+} \mu(t)/(t-t_0) = \inner{B\eta, \eta}$$
where $\eta$ obeys $(A \dotplus t_0 B)\eta = 0, \|\eta\| = 1$. (In particular, if $0$ is not an eigenvalue of $A$, then $\lim_{t \rightarrow t_0^+} \mu(t)/(t-t_0) \ne 0$ if and only if $0$ is an eigenvalue of $A \dotplus t_0 B$.)  
\end{theorem}

Note that the operator sum $A+tB$ may not be self-adjoint if $B$ is not bounded, so the expression $A\dotplus tB$ in the theorem above is meant to denote a self-adjoint extension of the operator sum of $A$ and $tB$. Such an extension is always possible, and in fact forms a self-adjoint holomorphic family of type (B). In this paper, we generalize Theorem \ref{thm:Simon} to any self-adjoint holomorphic family of type (B), denoted $A(t)$. Our results describe the rates at which the eigenvalues of $A(t)$ approach the minimum of the essential spectrum of $A(t_0)$ from below as $t \rightarrow t_0^+$. Specifically, we prove that there is a bounded self-adjoint operator $B_0$ such that if $\lambda(t)$ is an eigenvalue of $A(t)$ below the essential spectrum that depends continuously on $t$ when $t_0 < t < \epsilon$, if $\Sigma(t)$ is the minimum of the essential spectrum of $A(t)$ for any $t$, and if $\lim_{t \rightarrow t_0^+} \lambda(t) = \Sigma(t_0)$, then  $\lim_{t \rightarrow t_0^+}(\lambda(t) - \Sigma(t_0))/(t-t_0)$ is an eigenvalue of $B_0$. We also give a sufficient condition for the minimum of the essential spectrum to be an eigenvalue of $A(t_0)$ at a threshold value $t_0$ where eigenvalues from the discrete spectrum are absorbed by the essential spectrum. Our results also generalize some other previously known results, see \cite[Theorem 1.1]{KlausSimon80} and \cite[Theorem 3.1]{Simon77}.  

In order to prove these results, we use the essential numerical range of an auxilliary sesquilinear form determined by the type (B) family $A(t)$. The essential numerical range was introduced by Stampfli and Williams for bounded operators in \cite{StWi68}.  Recently, B\"{o}gli, Marletta, and Tretter extended the notion of essential numerical range to unbounded operators and sesquilinear forms \cite{BoMaTr19}. They also apply it to the problem of spectral pollution.  Their extension of the essential numerical range to unbounded operators is what allows us to generalize Theorem \ref{thm:Simon}. 

The paper is organized as follows. In Section \ref{sec:We}, we introduce the essential numerical range for operators and for sesquilinear forms. We also prove that the essential numerical range of a densely defined unbounded operator is either empty or unbounded, which appears to be a new result. In Section \ref{sec:perturb}, we review type (B) holomorphic families of self-adjoint operators and their properties.  Our main results are detailed in Section \ref{sec:main}.  Section \ref{sec:isolated} contains a minor observation about analytic perturbations of isolated eigenvalues with infinite multiplicity, and we conclude with examples in Section \ref{sec:ex}.

Throughout the paper, we will use the following notation.  $\HH$ will denote a complex Hilbert space.  All operators are assumed to act on $\HH$, unless stated otherwise.  For an operator $T$, the domain of $T$ will be denoted $\D(T)$ and the spectrum of $T$ is $\sigma(T)$.  If $x_n \in \HH$ converges to $x$ in norm, we write $x \rightarrow x$ and if $x_n$ converges weakly to $x$, then we write $x_n \xrightarrow{w} 0$.  For a sesquilinear form $\mf{t}$ on $\HH$, we write $\mf{t}[x]$ as shorthand for $t[x,x]$.  We denote the closure of a set $A$ by $\cl{A}$ and the convex hull by $\conv A$.  
\color{black}

\section{The Essential Numerical Range} \label{sec:We}
For an operator $T$ with domain $\mathcal{D}(T) \subseteq \mathcal{H}$, the \emph{numerical range} of $T$ is the set 
$$W(T) = \{ \inner{Tx,x} \, : \, x \in \D(T) \text{ with } \|x\|=1\}.$$
According to the Toeplitz-Hausdorff theorem, the numerical range of an operator is always a convex set. The \emph{essential numerical range} of $T$ is the set
$$\We(T) = \{\lambda \in \C : \exists \, x_n \in \mathcal{D}(T) \text{ with } \|x_n\|=1, x_n \xrightarrow{w} 0, \inner{Tx_n,x_n} \rightarrow \lambda \}.$$
There are several alternative definitions of essential spectrum, but we will follow \cite{BoMaTr19} and use the following:
$$\sige(T) = \{\lambda \in \C : \exists \, x_n \in \mathcal{D}(T) \text{ with } \|x_n\|=1, x_n \xrightarrow{w} 0, \|(T-\lambda)x_n\| \rightarrow 0 \}.$$

The following lemma collects some of the basic properties of the essential numerical range. The first two assertions are \cite[Proposition 2.2]{BoMaTr19}, while the third follows immediately from the definition of $\We(T)$. 
\begin{lemma} \label{lem:basics}
Let $T$ be an operator (unbounded or bounded) on $\HH$. Then 
\begin{enumerate}
\item $\We(T)$ is closed and convex.
\item $\sige(T) \subseteq \We(T)$
\item $\We(cT) = c\We(T)$ for all $c \in \C$.
\end{enumerate}
\end{lemma}
If $A$ is a self-adjoint operator on $\HH$ that is bounded below, then $\min \sige(A) = \min \We(A)$ \cite[Theorem 3.8]{BoMaTr19}.

The following result shows that pre-images of points in the numerical range have a kind of continuity property. We need it to prove an important fact about the essential numerical range. A weaker version appeared as \cite[Theorem 2]{CoJoKiLiSp13} for matrices in $\C^{n \times n}$ and for bounded linear operators as \cite[Proposition 3.2]{LiSp20}.
\begin{theorem} \label{thm:cap}
Let $T$ be an operator on $\HH$.  Let $z = \inner{Tx,x}$ for some $x \in \D(T)$ with $\|x\|=1$. For any $0 < \epsilon < 1$, the set 
$$\{\inner{Ty,y} \,:\, y \in \D(T), ~ \|y\|=1, \text{ and } |\inner{x,y}|^2 \ge 1-\epsilon\}$$
contains $\epsilon W(T) + (1-\epsilon) z$.  
\end{theorem}
\begin{proof}
It suffices to prove the theorem for the compression of $T$ onto any two dimensional subspace of $\D(T)$ that contains $x$.  Therefore we will assume that $\HH = \C^2$. In that case, Davis \cite{Davis71} observed that the numerical range of $T$ is the image of the set $S = \{yy^* \, : \, y \in \C^2, ~\|y \|=1\}$ under the linear transformation $F:X \mapsto \tr(TX)$, and the set $S$ is a sphere contained in the 3-dimensional real affine space of $2$-by-$2$ Hermitian matrices with trace equal to one.   

The set of 2-by-2 Hermitian matrices has inner product $\inner{X,Y} = \tr(XY)$. For $x, y \in \C^2$, 
$$\inner{xx^*,yy^*} = \tr(xx^*yy^*) = x^*y \tr(xy^*) = | \inner{x,y} |^2.$$
Let $C = \{ yy^* \, : \, y \in \C^2, ~\|y \| = 1, \text{ and } | \inner{x,y} |^2 \ge 1 - \epsilon \}$. Thus $C$ is the spherical cap formed by intersecting $S$ with the half-space $H$ consisting of all 2-by-2 Hermitian matrices $Y$ such that $\inner{Y,xx^*} \ge 1-\epsilon$.  
The set $\epsilon S + (1-\epsilon) xx^*$ is contained in both the convex hull of $S$ and in the half-space $H$, therefore it is contained in the convex hull of $C$. It is not hard to show that the image of the spherical cap $C$ under the linear transformation $F$ is a convex set \cite[Lemma 3]{CoJoKiLiSp13}. Therefore 
\begin{align*}
\epsilon W(T) + (1-\epsilon) z &= F(\epsilon S + (1-\epsilon) xx^*) \subset F(\conv C) = F(C) \\
&= \{\inner{Ty,y} \,:\,  y \in \C^2, ~ \|y\|=1, \text{ and } |\inner{x,y}|^2 \ge 1-\epsilon\}.
\end{align*}
\end{proof}

Theorem \ref{thm:cap} lets us prove the following generalization of \cite[Proposition 2.4]{BoMaTr19}. 
\begin{theorem} \label{thm:rays}
Let $T$ be an unbounded operator on $\mathcal{H}$. Suppose $w \in \We(T)$ and $z_n$ is a sequence in $W(T)$ such that $|z_n| \rightarrow \infty$. If $\frac{z_n - w}{|z_n -w|}$ has a limit point $v$, then the ray $\{w + tv : t \ge 0 \}$ is contained in $\We(T)$.
\end{theorem}
\begin{proof}
Since $w \in \We(T)$, there is a sequence $x_n \in \D(T)$ with $\|x_n\| =1$, $\inner{Tx_n,x_n} \rightarrow w$ and $x_n \xrightarrow{w} 0$.  Let $w_n = \inner{Tx_n,x_n}$.  For any $t > 0$, let $\epsilon_n = t/|z_n - w_n|$. Then $\epsilon_n \rightarrow 0$ as $n \rightarrow \infty$. We can assume that $\epsilon_n < 1$ for all $n$. Observe that 
$$w_n + t \tfrac{z_n-w_n}{|z_n-w_n|} = w_n + \epsilon_n (z_n-w_n) \in \epsilon W(T) + (1-\epsilon)w_n$$ 
for all $n$.  By Theorem \ref{thm:cap}, there is a $y_n \in \D(T)$ with $\|y_n\| = 1$ such that $\inner{Ty_n,y_n} = w_n + t \frac{z_n-w_n}{|z_n-w_n|}$ and $|\inner{x_n,y_n}|^2 \ge 1-\epsilon_n$.  Since $\inner{T(\omega y_n),\omega y_n} = \inner{Ty_n,y_n}$ for all unimodular constants $\omega \in \C$, we may assume that $\inner{x_n, y_n} > 0$.  Then $\|x_n-y_n\|^2 \le 2 - 2 \sqrt{1-\epsilon_n}$. This means that $\|x_n -y_n\| \rightarrow 0$ and therefore $y_n \xrightarrow{w} 0$.  Then, since $\inner{Ty_n,y_n} = w_n + t \frac{z_n-w_n}{|z_n-w_n|}$ has a subsequence converging to  $w + tv$, we conclude that $w+tv \in \We(T)$.
\end{proof}

One consequence of Theorem \ref{thm:rays} is that if $\We(T)$ is nonempty and $W(T)$ is unbounded, then $\We(T)$ is also unbounded. It is possible, however, for $W(T)$ to be unbounded and $\We(T)$ to be empty, see \cite[Example 2.6]{BoMaTr19}. It is well known that the numerical range of a densely defined unbounded operator is unbounded. This can be seen from the polarization identity 
\begin{align*}
\inner{Tx,y} &= \tfrac{1}{4}(\inner{T(x+y),x+y} - \inner{T(x-y),x-y} \\ &+ i\inner{T(x+iy),x+iy} - i\inner{T(x-iy),x-iy}) \text{ for all } x, y \in \D(T).
\end{align*}
Therefore Theorem \ref{thm:rays} implies the following.

\begin{corollary}
If $T$ is a densely defined unbounded operator on a complex Hilbert space $\HH$, then $\We(T)$ is either empty or unbounded.
\end{corollary}

The essential numerical range can also be defined for sesquilinear forms. Let $\mathfrak{t}$ be a sesquilinear form on $\HH$ with domain $\D(\mathfrak{t})$. The numerical range $W(\mathfrak{t})$ and the essential numerical range $\We(\mathfrak{t})$ are defined as
$$W(\mathfrak{t}) = \{ \mathfrak{t}[x] \, : \, x \in \D(\mathfrak{t}), \|x \| = 1 \}$$
and 
$$\We(\mathfrak{t}) = \{ z \in \C \, : \, \exists \, x_n \in \D(\mathfrak{t}) \text{ with } \|x_n \| = 1, x_n \xrightarrow{w} 0, \mathfrak{t}[x_n] \rightarrow z \}.$$
The definition of the essential numerical range for forms was introduced in \cite{BoMaTr19}. Note that both the numerical range and the essential numerical range of a sesquilinear form are convex, and the essential numerical range is closed.  In addition, Theorem \ref{thm:rays} applies to $\We(\mathfrak{t})$ for any sesquilinear form $\mathfrak{t}$, and the proof is the same. 

For a sesquilinear form $\mathfrak{t}$, the adjoint $\mathfrak{t^*}$ is defined by
$$\mathfrak{t^*}[x,y] = \overline{\mathfrak{t}[y,x]}$$
for all $x, y \in \D(\mathfrak{t})$ and $\D(\mathfrak{t^*}) = \D(\mathfrak{t})$.  A sesquilinear form $\mf{t}$ is \emph{symmetric} if $\mf{t} = \mf{t^*}$. The real and imaginary parts of a form $\mathfrak{t}$ are
$$\re \mathfrak{t} = \tfrac{1}{2} (\mathfrak{t}+ \mathfrak{t^*}) \text{ and }  \im \mathfrak{t} = \tfrac{1}{2i} (\mathfrak{t} - \mathfrak{t^*}).$$
Unlike with unbounded operators, it is always possible to define the real and imaginary part of a sesquilinear form, and the domains of the real and imaginary parts are the same as the domain of the original form.

A sesquilinear form $\mathfrak{t}$ is \emph{sectorial} if there is a vertex $v \in \C$ and an angle $\theta < \pi/2$ such that $|\arg(z-v)| \le \theta$ for all $z \in W(\mathfrak{t})$.  Likewise, an operator $T$ on $\HH$ is \emph{sectorial} if the sesquilinear form $\mathfrak{q}_T$ defined by $\mathfrak{q}_T[x,y] = \inner{Tx,y}$ with domain $\D(\mf{q}_T) = \D(T)$ is sectorial. The following lemma is essentially \cite[Theorem 3.10(ii)]{BoMaTr19}, although there it is stated for $m$-sectorial operators, rather than sectorial forms. The proof is the same, however, and we include it because it is short.
\begin{lemma} \label{lem:sectorial}
Let $\mathfrak{t}$ be a sectorial sesquilinear form on $\HH$.  Then $\re \We(\mathfrak{t}) = \We(\re \mathfrak{t})$.  
\end{lemma}
\begin{proof}
If $z \in \We(\mathfrak{t})$, then there is a sequence $x_n \in \D(\mathfrak{t})$, $\|x_n\| = 1$, $x_n \xrightarrow{w} 0$ such that $\mf{t}[x_n] \rightarrow z$. Then $(\re{\mf{t}})[x_n] \rightarrow \re z$ so $\re z \in \We(\re \mf{t})$. This proves that $\re \We(\mf{t}) \subseteq \We(\re \mf{t})$. Conversely, suppose that $a \in \We(\re \mf{t})$.  Let $x_n \in \D(\mathfrak{t})$ be a sequence with $\|x_n\|=1$ and $x_n \xrightarrow{w} 0$ such that $(\re \mf{t})[x_n] \rightarrow a$.  Since $\mf{t}$ is sectorial, $(\im \mf{t})[x_n]$ is bounded.  We may pass to a subsequence such that $(\im \mf{t})[x_n]$ converges to some $b$, and then $\mf{t}[x_n] \rightarrow a+ib \in \We(\mf{t})$.  Therefore $\We(\re \mf{t}) \subseteq \re \We(\mf{t})$.
\end{proof}

Suppose that $\mf{t}$ is a sectorial form on $\HH$.  We say that $\mf{t}$ is \emph{closed} if for any sequence $x_n \in \D(\mf{t})$ such that $x_n$ converges to $x \in \HH$ and $\mf{t}[x_n-x_m] \rightarrow 0$ when $n, m \rightarrow \infty$, we have $x \in \D(\mf{t})$.  A sectorial form is \emph{closable} if it has a closed extension. The \emph{closure} of a sectorial form $\mf{s}$ is the smallest closed extension $\mf{t}$, and its domain is the set of $x,y \in \HH$ such that there exist sequences $x_n, y_n \in \D(\mf{s})$ such that $x_n \rightarrow x$, $y_n \rightarrow y$, and $\mf{s}[x_n - x_m] \rightarrow 0$, $\mf{s}[y_n-y_m] \rightarrow 0$ as $m,n \rightarrow \infty$. For any such $x$ and $y$, $\mf{t}$ is defined by
$$\mf{t}[x, y] = \lim_{n \rightarrow \infty} \mf{s}[x_n, y_n].$$ 
If $\mf{t}$ is a closed sesquilinear form, then a subspace $V \subset \D(\mf{t})$ is a \emph{core} of $\mf{t}$ if the closure of the restriction of $\mf{t}$ to $V$ is $\mf{t}$.  See \cite[Chapter VI]{Kato} for details. Note that if $\mf{t}$ is the closure of a sesquilinear form $\mf{s}$, then $\We(\mf{t}) = \We(\mf{s})$. This was observed in \cite{BoMaTr19}, and is easy to verify from the definition above. 

The following technical lemma will be needed for the proof of our main result, Theorem \ref{thm:weakerMain}. 
\begin{lemma} \label{lem:technical}
Let $\mf{a}$ be a densely defined, closed, symmetric sesquilinear form on $\HH$ that is bounded below.  Let $x, y \in \D(\mf{a})$. Suppose that $x_n \in \D(\mf{a})$ is a sequence that weakly converges to $x$ and $\mf{a}[x_n]$ is bounded.  Then $\lim_{n \rightarrow \infty} \mf{a}[x_n,y] = \mf{a}[x,y]$.   
\end{lemma}

\begin{proof}
Since $\mf{a}$ is bounded below, there is a constant $c \in \R$ such that $\mf{a}+c$ is nonnegative. Then, by the second representation theorem for sesquilinear forms \cite[Theorem VI.2.23]{Kato}, there is a nonnegative self adjoint operator $A$ such that $\mf{a}[x,y] = \inner{A^{1/2}x,A^{1/2}y} - c\inner{x,y}$ for all $x,y \in \D({\mf{a}})$ and $\D({\mf{a}}) = \D(A^{1/2})$.  Since $\|x_n\|$ and $\mf{a}[x_n]$ are bounded, it follows that $\|A^{1/2} x_n\|$ is bounded.  So $A^{1/2} x_n$ must have a weak limit point $u \in \HH$. 
Since the graph of $A^{1/2}$ is closed in the weak topology, the pair $(x,u)$ must be in the graph, so $u = A^{1/2}x$ and $A^{1/2} x_n \xrightarrow{w} A^{1/2}x$.  
This implies that $\mf{a}[x_n,y] = \inner{A^{1/2}x_n,A^{1/2}y} - c \inner{x_n,y}$ converges to $\inner{A^{1/2}x,A^{1/2}y} - c \inner{x,y} = \mf{a}[x,y]$.
\end{proof}

\section{Perturbations of Self-Adjoint Operators} \label{sec:perturb}

Let $A$ be a self-adjoint operator on $\HH$ that is bounded below. The sesquilinear form $\mf{q}_A$ defined by $\mf{q}_A[x,y]= \inner{Ax,y}$ with domain $\D(\mf{q}_A) =\D(A)$ is closable \cite[Theorem VI.1.27]{Kato}.  Let $\mathfrak{a}$ denote its closure.  Then $\mathfrak{a}$ is a closed, densely defined, symmetric form that is bounded below. If $\xi < \min \sigma(A)$, then $\D(\mf{a}) = \D((A- \xi)^{1/2})$, and 
$$\mf{a}[x,y] = \inner{(A-\xi)^{1/2}x,(A-\xi)^{1/2}y} + \xi \inner{x,y}$$ 
for all $x,y \in \D(\mf{a})$. This follows immediately from the second representation theorem for sesquilinear forms \cite[Theorem VI.2.23 and Problem VI.2.25]{Kato}. We will refer to the form $\mf{a}$ as the closed sesquilinear form corresponding to $A$.  The correspondence also works the other way. That is, for any densely defined, closed, symmetric, sesquilinear form $\mf{a}$ there is a unique self-adjoint operator $A$ such that $\D(A)$ is a core for $\mf{a}$ \cite[Theorems VI.2.1 and VI.2.6]{Kato}.  In particular the essential numerical range of $A$ and $\mf{a}$ are the same, so $\min \sige(A) = \min \We(A) = \min \We(\mf{a})$. 

A family $\mf{a}(t)$ of sesquilinear forms defined in a neighborhood of $0$ is called \emph{holomorphic of type (a)} if each $\mf{a}(t)$ is closed and sectorial, the domain $\D(\mf{a}(t))$ is constant and dense in $\HH$, and $\mf{a}(t)[x,y]$ is a holomorphic function of $t$ in a neighborhood of 0 for all $x, y \in \D(\mf{a}(t))$. Since the domain is constant, we'll write $\D(\mf{a})$ in place of $\D(\mf{a}(t))$. If each $\mf{a}(t)$ is symmetric when $t$ is real, then for every real $t$ in a neighborhood of zero there is a self-adjoint operator $A(t)$ such that $\inner{A(t)x,y} = \mf{a}(t)[x,y]$ for all $x, y \in \D(A(t))$, and the domain of each $A(t)$ is a core for $\D(\mf{a})$.  This family $A(t)$ is a \emph{holomorphic family of type (B)} \cite{Kato}. 

\begin{lemma} \label{lem:powerSeries}
Let $A(t)$ be a self-adjoint holomorphic family of type (B).  Let $\mf{a}(t)$ be the corresponding type (a) family of sesquilinear forms. There is a constant $R > 0$ and a sequence of symmetric sesquilinear forms $\mf{a}_k$ such that 
\begin{equation} \label{eq:powerSeries}
\mf{a}(t)[x,y] = \mf{a}_0[x,y] + t \mf{a}_1 [x,y] + t^2 \mf{a}_2[x,y] + \ldots
\end{equation}
for all $|t| < R$. Each $\mf{a}_k$ has $\D(\mf{a}) \subseteq \D(\mf{a}_k)$ and for any $0 < r < R$, there are constants $b, c \ge 0$ such that 
\begin{equation} \label{eq:Bbound}
|\mf{a}_k[x,y]| \le \frac{c}{r^k} \sqrt{\mf{a}_0[x]+b \|x\|^2} \sqrt{\mf{a}_0[y]+b \|y\|^2}
\end{equation}
for all $x, y \in \D(\mf{a})$ and $k \in \N$. In particular, the series \eqref{eq:powerSeries} converges absolutely for all $x, y \in \D(\mf{a})$ when $|t| < R$.  
\end{lemma}

\begin{proof}
The self-adjoint operator $A = A(0)$ is bounded below, so we can choose $b > 0$ such that $-b < \min \sigma(A)$. Then $(A+bI)^{-1/2}$ is a bounded self-adjoint operator. The family $B(t) = (A+bI)^{-1/2} A(t) (A+bI)^{-1/2}$ is a bounded-holomorphic family of self-adjoint operators (see the proof of \cite[Theorem VII.4.2]{Kato}).  Therefore it has a series expansion which converges absolutely in norm for all $t$ such that $|t| < R$ for some $R > 0$:
$$B(t) = B_0 + t B_1 + t^2 B_2 + \ldots.$$
Then 
\begin{align*}
\mf{a}(t)[x,y] &= \inner{B(t) (A+bI)^{1/2} x, (A+bI)^{1/2} y} \\
&= \sum_{k = 0}^{\infty} t^k \inner{B_k (A+bI)^{1/2} x, (A+bI)^{1/2} y}
\end{align*}
for any $x,y \in \D(\mf{a})$.  Let $\mf{a}_k[x,y] = \inner{B_k (A+bI)^{1/2} x, (A+bI)^{1/2} y}$ for all $x, y \in \D(\mf{a})$. Let $0 < r < R$. Since the series for $B(t)$ converges absolutely in norm when $|t| < R$, there is a constant $c > 0$ such that $\|B_k\| \le c/r^k$ for all $k$. Observe that $\|(A+bI)^{1/2}v\|^2 = \mf{a}_0[v] + b \|v\|^2$ for any $v \in \D(\mf{a})$. Therefore \eqref{eq:Bbound} follows from the Cauchy-Schwarz inequality. 
\end{proof}

As observed in the proof, the constant $b$ in \eqref{eq:Bbound} can be any value such that $-b < \min \sigma(A)$.  It will be convenient for a type (B) family $A(t)$ to fix such a constant $b$, and introduce the norm $\|x\|_\mf{a} = \sqrt{\mf{a}_0[x] + b\|x\|}$.  Then \eqref{eq:Bbound} becomes
$$|\mf{a}_k[x,y]| \le \frac{c}{r^k} \|x\|_\mf{a} \|y \|_\mf{a}.$$
Since the series in Lemma \ref{lem:powerSeries} converges absolutely when $|t| < R$, if we fixed $0 < r < R$, then there are constants $M_1, M_2 > 0$ such that 
\begin{equation} \label{eq:M1}
|\mf{a}(t)[x,y] - \mf{a}_0[x,y]| \le M_1 t \|x\|_\mf{a} \|y\|_\mf{a}
\end{equation}
and 
\begin{equation} \label{eq:M2}
|\mf{a}(t)[x,y] - (\mf{a}_0+t\mf{a}_1)[x,y]| \le M_2 t^2 \|x\|_\mf{a} \|y\|_\mf{a}
\end{equation}
for all $x,y \in \D(\mf{a})$ when $|t| \le r < R$. If $t> 0$ is sufficiently small so that $M_1 t$ and $M_1 t+ M_2 t^2$ are both less than $\tfrac{1}{2}$, then \eqref{eq:M1} and \eqref{eq:M2} imply that 
\begin{equation} \label{eq:uniform}
\begin{aligned}
\tfrac{1}{2} \|x\|_\mf{a} &\le \mf{a}(t)[x]+b\|x\|^2 \le \tfrac{3}{2} \|x\|_\mf{a} \text{ and } \\
\tfrac{1}{2} \|x\|_\mf{a} &\le (\mf{a}+t\mf{a}_1)[x]+b\|x\|^2 \le \tfrac{3}{2} \|x\|_\mf{a} 
\end{aligned}
\end{equation}
for any $x \in \D(\mf{a})$.

The next result connects the minimum of the essential spectrum of $A(t)$ with the essential numerical range of an auxiliary sesquilinear form.   
\begin{proposition} \label{prop:bottom}
Let $A(t)$ be a self-adjoint holomorphic family of type (B) on a complex Hilbert space $\HH$. 
Let $\Sigma(t) = \min \sige(A(t))$ and assume that $\Sigma(0) = 0$, let $\mf{a}(t)$ be the family of sesquilinear forms corresponding to $A(t)$, and let $\mathfrak{t} = \mathfrak{a}_0 + i\mathfrak{a}_1$ where $\mf{a}_0, \mf{a}_1$ are defined as in Lemma \ref{lem:powerSeries}.  Then $\lim_{t \rightarrow 0^+} \Sigma(t)/t = \omega$ where $\omega$ is the smallest value such that $i \omega \in \We(\mathfrak{t})$.   
\end{proposition}

\begin{proof}
Observe that $\Sigma(t) = \min \We(A(t)) = \min \We(\mathfrak{a}(t))$. The bound in \eqref{eq:Bbound} implies that $\mf{t}$ is sectorial, and also that $(1-it)\mf{t}$ is sectorial when $0 < t < \epsilon$ as long as $\epsilon$ is chosen sufficiently small. So 
$$\re \We((1-it) \mathfrak{t}) = \We(\re ((1-it) \mathfrak{t})) = \We(\mathfrak{a}_0+t \mathfrak{a}_1)$$ 
by Lemma \ref{lem:sectorial}. Let $\Sigma_1(t)$ denote the minimum of $\We(\mathfrak{a}_0+t \mathfrak{a}_1)$. Of course, $\Sigma_1(0) = \Sigma(0) = 0$, and $\We(\mathfrak{t})$ has nonempty intersection with the imaginary axis.  Let $\omega \in \R$ be the minimum value such that $i \omega \in \We(\mathfrak{t})$. 

For any $\omega' < \omega$, there is a line $\ell$ passing through $i \omega'$ that does not intersect $\We(\mathfrak{t})$, see Figure \ref{fig:rotate}. Let $K$ be the closed convex set in $\C$ containing all points to the right of the imaginary axis and above the line $\ell$. In particular, $K$ contains $\We(\mathfrak{t})$.  As long as $t > 0$ is small enough so that $(1-it)\ell$ is not vertical, the leftmost point in $(1-it)K$ is $t\omega' + i \omega'$. Thus, $\re((1-it)K) = [t \omega',\infty)$.  We then have 
$$\Sigma_1(t) \in \re \We((1-it)\mathfrak{t}) \subset \re((1-it)K) = [t \omega', \infty).$$
Therefore, $\Sigma_1(t) \ge t \omega'$ for all $t>0$ sufficiently small.  

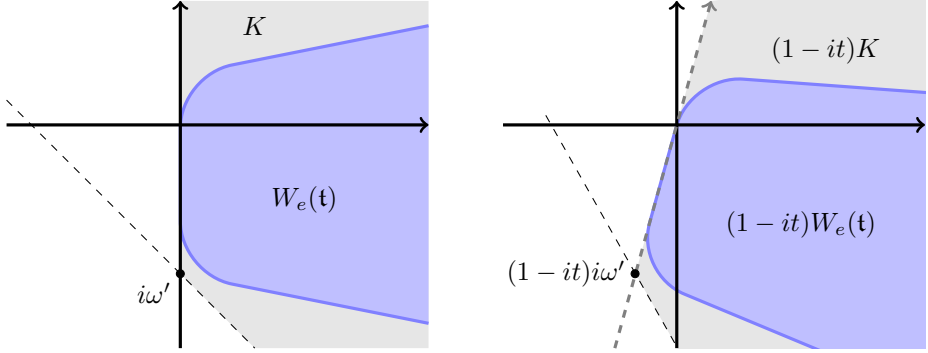
\begin{figure}[h]
\begin{center}
\begin{tikzpicture}[scale=0.66]
\begin{scope}
\fill[fill=black!10] (5,-4.5) -- (1.5,-4.5) -- (0,-3) -- (0,2.5) -- (5,2.5) -- cycle;
\filldraw[fill=blue!25,very thick, draw=blue!50, rounded corners=0.7cm] (5,-4) -- (0,-3) -- (0,1) -- (5,2);
\draw[very thick,->] (-3.5,0) -- (5,0);
\draw[very thick,->] (0,-4.5) -- (0,2.5);
\draw[dashed] (-3.5,0.5) -- (1.5,-4.5);
\draw (1.5,2) node {$K$};
\filldraw (0,-3) circle (0.08) node[below left] {$i \omega'$};
\draw (2.5,-1.5) node {$\We(\mf{t})$};
\end{scope}

\begin{scope}[xshift=10cm]
\clip (-3.7,-4.5) rectangle (5,2.5);
\fill[fill=black!10,,rotate=-15.5,scale=1.03786] (5,-4.5) -- (1.5,-4.5) -- (0,-3) -- (0,2.5) -- (5,4.5) -- cycle;
\filldraw[fill=blue!25,very thick, draw=blue!50, rounded corners=0.7cm,rotate=-15.5,scale=1.03786] (8,-4) -- (0,-3) -- (0,1) -- (5,2);
\draw[very thick,->] (-3.5,0) -- (5,0);
\draw[very thick,->] (0,-4.5) -- (0,2.5);
\draw[very thick,gray,->,dashed,rotate=-15.5,scale=1.03786] (0,-4.5) -- (0,2.5);
\draw[dashed,rotate=-15.5,scale=1.03786] (-2.5,-0.5) -- (1.5,-4.5);
\filldraw (-0.84,-3) circle (0.08) node[left=0.1] {$(1-it)i \omega'$};
\draw (2.5, -2) node {$(1-it)\We(\mf{t})$};
\draw (3,1.5) node {$(1-it)K$};
\end{scope}
\end{tikzpicture}
\end{center}
\caption{$i\omega'$ can be separated from $\We(\mf{t})$ (left). So the real part of $\We((1-it)\mf{t})$ is bounded below for small $t$ (right).}
\label{fig:rotate}
\end{figure}

At the same time, $t \omega \in \re \We((1-it) \mathfrak{t})$, so $\Sigma_1(t) \le t \omega$ for all $0 < t < \epsilon$.  Therefore all limit points of $\Sigma_1(t)/t$ lie between $\omega$ and $\omega'$.  But since $\omega'$ can be chosen arbitrarily close to $\omega$, we conclude that $\lim_{t \rightarrow 0^+} \Sigma_1(t)/t = \omega$. 

As in the comments after Lemma \ref{lem:powerSeries}, choose $b > 0$ so that $-b < \sigma(A)$. Define $\|x\|_\mf{a} = \sqrt{\mf{a}_0[x]+b\|x\|^2}$.  Since $\Sigma(0) = 0$, there is a sequence $x_n \in \D(\mf{a})$, $x_n \xrightarrow{w} 0$, $\|x_n\|=1$, such that $\mf{a}_0[x_n] \rightarrow 0$.  That means $\|x_n\|_\mf{a} \rightarrow 1$, and so it follows from \eqref{eq:M1} and \eqref{eq:M2} that $\mf{a}(t)[x_n]$ and $(\mf{a}_0+t \mf{a}_1)[x_n]$ are bounded.  We can take a subsequence of $x_n$ such that both $\mf{a}(t)[x_n]$ and $(\mf{a}_0+t \mf{a}_1)[x_n]$ converge to points in $w \in \We(\mf{a}(t))$ and $z \in \We(\mf{a}_0+t \mf{a_1})$, respectively.  As long as $t$ is sufficiently small, \eqref{eq:uniform} implies that $z, w < \frac{3}{2}-b$ so $\Sigma(t) \le \tfrac{3}{2} - b$ and $\Sigma_1(t) \le \frac{3}{2} -b$, as well.  

Now, take a sequence $x_n \in \D(\mf{a})$, $x_n \xrightarrow{w} 0$, $\|x_n\|=1$ such that $\mf{a}(t)[x_n] \rightarrow \Sigma(t)$.  Observe that $\limsup_{n \rightarrow \infty} \|x_n\|_\mf{a} \le 2(\Sigma(t)+b)$ by \eqref{eq:uniform}, so $(\mf{a}_0+t\mf{a}_1)[x_n]$ is bounded and we can pass to a subsequence such that $(\mf{a}_0+t\mf{a}_1)[x_n]$ converges to $z \in \We(\mf{a}_0+t\mf{a_1})$. By definition, $\Sigma_1(t) \le z$. Applying \eqref{eq:M2} to the terms $\mf{a}(t)[x_n]$ and $(\mf{a}_0+t\mf{a}_1)[x_n]$ and taking the limit, we see that there is a constant $C > 0$ such that $|\Sigma(t) - z| \le C t^2$ for all $t$ sufficiently small.  Similarly, if we take a sequence $y_n \in \D(\mf{a})$ such that $y_n \xrightarrow{w} 0$, $\|y_n \| =1$, and $(\mf{a}_0+t\mf{a}_1) y_n \rightarrow \Sigma_1(t)$, then we have a subsequence of $y_n$ such that $\mf{a}(t)[y_n] \rightarrow w \in \We(\mf{a}(t))$ and $|\Sigma_1(t) - w| \le C t^2$.  If $\Sigma_1(t) \ge \Sigma(t)$, then $\Sigma_1(t) - \Sigma(t) \le z - \Sigma(t) \le C t^2$.  If, on the other hand, $\Sigma(t) \ge \Sigma_1(t)$, then $\Sigma(t)-\Sigma_1(t) \le w - \Sigma_1(t) \le C t^2$.  Either way, we conclude that $|\Sigma(t) - \Sigma_1(t) | \le C t^2$ when $t$ is small, and therefore $\lim_{t \rightarrow 0^+} \Sigma(t)/t = \lim_{t \rightarrow 0^+} \Sigma_1/t = \omega$.  
\end{proof}

\section{Main Results} \label{sec:main}

Let $A(t)$ be a self-adjoint holomorphic family of type (B) on a complex Hilbert space $\HH$. Let $\mf{a}(t)$ be the type (a) family of sesquilinear forms corresponding to $A(t)$, and suppose that $\mf{a}(t)$ has series expansion $\mf{a}(t) = \mf{a}_0 + t\mf{a}_1 + t^2 \mf{a}_2 + \ldots$ given by Lemma \ref{lem:powerSeries}. If the kernel of $A = A(0)$ is nontrivial, let $P$ denote the orthogonal projection onto the kernel of $A$. Since the quadratic form defined by $x \mapsto \mathfrak{a}_1[Px]$ is bounded and symmetric, there is a bounded self-adjoint operator $B_0$ defined by the sesquilinear form
\begin{equation} \label{eq:B0}
\inner{B_0 x,y} = \mathfrak{a}_1[Px,Py] 
\end{equation}
for all $x, y \in \HH$ \cite[Section V.2.1]{Kato}. With this observation, we are now ready to state our main result. 

\begin{theorem} \label{thm:weakerMain}
Let $A(t)$ be a self-adjoint holomorphic family of type (B). 
Let $\Sigma(t) = \min \sige(A(t))$ and assume that $\Sigma(0) = 0$. If $\lambda(t) < \Sigma(t)$ is an eigenvalue of $A(t)$ that depends continuously on $t$ when $0 < t < \epsilon$, $\lim_{t \rightarrow 0^+} \lambda(t) = 0$, and $\liminf_{t \rightarrow 0^+} (\lambda(t) - \Sigma(t))/t < 0$, then $0$ is an eigenvalue of $A(0)$.  Furthermore, $\lim_{t \rightarrow 0^+} \lambda(t)/t$ exists and  is equal to an eigenvalue of the self-adjoint operator $B_0$ defined by \eqref{eq:B0}.   
\end{theorem}

\begin{proof}
Let $\omega = \lim_{t \rightarrow 0^+} \Sigma(t)/t$ and $\beta  = \liminf_{t \rightarrow 0^+} \lambda(t)/t$. So $\beta < \omega$ by assumption. Note that the limit $\omega$ exists by Proposition \ref{prop:bottom}.  Let $A = A(0)$ and let $\mf{a}(t)$ be the type (a) family of sesquilinear forms corresponding to $A(t)$. Suppose that $\mf{a}(t)$ has series expansion $\mf{a}(t) = \mf{a}_0 + t\mf{a}_1 + t^2 \mf{a}_2 + \ldots$ given by Lemma \ref{lem:powerSeries}.

By the Kato-Rellich theory, $\lambda(t)$ is an analytic function of $t$, except possibly at isolated points $0 < t < \epsilon$ where $\lambda(t)$ crosses another eigenvalue function. For each $t$ where $\lambda(t)$ does not have a crossing, any unit eigenvector $x(t)$ of $A(t)$ corresponding to $\lambda(t)$ will have $\lambda'(t) = \mathfrak{a}_1[x(t)]$ \cite[Problem VII.4.19]{Kato}.

If $\lambda(t)/t$ has more than one limit point as $t \rightarrow 0^+$, then $\lambda(t)/t$ will oscillate between $\beta$ and a value greater than $\beta$ infinitely many times, so it will be possible to choose a sequence $t_n \rightarrow 0^+$ such that $\lambda(t_n)/t_n \rightarrow \beta$ and $\lambda'(t_n) \rightarrow \beta' \le \beta$ by the mean value theorem.

Suppose, on the other hand, that $\lambda(t)/t$ converges to $\beta$. Observe that 
$$\frac{d}{dt} \frac{\lambda(t)}{t} = \frac{1}{t} \left( \lambda'(t) - \frac{\lambda(t)}{t} \right),$$
and therefore for any $0 < t < t_0$, we have 
$$\frac{\lambda(t_0)}{t_0} - \frac{\lambda(t)}{t} = \int_t^{t_0} \frac{1}{\tau} \left( \lambda'(\tau) - \frac{\lambda(\tau)}{\tau} \right) \, d\tau.$$
Fix $\delta > 0$ and suppose that $t_0$ is small enough so that $|\lambda(t)/t - \beta| < \delta$ for every $0 < t \le t_0$.  If $\lambda'(t) > \beta + 2 \delta$ for almost every $0 < t < t_0$, then 
$$2 \delta > \frac{\lambda(t_0)}{t_0} - \frac{\lambda(t)}{t} = \int_t^{t_0} \frac{1}{\tau} \left( \lambda'(\tau) - \frac{\lambda(\tau)}{\tau} \right) \, d\tau \ge \delta \ln \left( \frac{t_0}{t} \right).$$
This is a contradiction, since $\ln(t_0/t) \rightarrow \infty$ as $t \rightarrow 0^+$.  Therefore, for every $t_0, \delta > 0$, it is possible to choose $t$ such that $0 < t < t_0$ and $\lambda'(t) < \beta + 2 \delta$.  So we can choose a sequence $t_n \rightarrow 0^+$ such that $\lambda(t_n)/t_n \rightarrow \beta$ and $\lim_{n \rightarrow \infty} \lambda'(t_n) = \beta'$ where $\beta' \le \beta+ 2 \delta < \omega$.  

Regardless of whether or not $\lambda(t)/t$ converges, we have shown that we can construct a sequence $t_n \rightarrow 0^+$ such that $\lambda(t_n)/t_n \rightarrow \beta$ and $\lambda'(t_n) \rightarrow \beta' < \omega$. For each $t_n$, let $\lambda_n = \lambda(t_n)$ and let $x_n$ be a unit eigenvector of $A(t_n)$ corresponding to $\lambda_n$. By \eqref{eq:M1} and \eqref{eq:uniform}, $\lim_{n \rightarrow \infty} \mf{a}_0[x_n] = 0$. Then $\lim_{n \rightarrow \infty} (\mf{a}_0+i\mf{a}_1)[x_n] = i \beta'$. Since $\omega$ is the smallest value such that $i \omega \in \We(\mf{a}_0+i\mf{a}_1)$ by Proposition \ref{prop:bottom}, we see that $i \beta'$ is outside $\We(\mf{a}_0+i\mf{a}_1)$.  By passing to a subsequence, we can assume that $x_n$ converges weakly to some $x \in \HH$.  Since $i \beta' \notin \We(\mf{a}_0+i \mf{a}_1)$, we conclude that $x \ne 0$.  

The domain of $A$ is a core for $\mf{a}_0$ and $\mf{a}_0[y] = \inner{Ay,y}$ for all $y \in \D(A)$. 
Therefore there is a sequence $y_n \in \D(A)$ such that $\|y_n\| =1$, $\|y_n - x_n \| < \frac{1}{n}$, and $| \inner{A y_n, y_n} - \mathfrak{a}_0[x_n] | < \frac{1}{n}$ for all $n$.  Then 
$$\lim_{n \rightarrow \infty} \inner{A y_n, y_n} =  \lim_{n \rightarrow \infty} \mathfrak{a}_0[x_n] = 0.$$
Let $A_\pm$ denote the positive and negative parts of $A$. For each negative eigenvalue $\nu$ of $A$ with multiplicity $m$, there is a corresponding analytic family of $m$ mutually orthogonal unit eigenvectors $v_j(t)$, $j = 1, \ldots, m$, of $A(t)$ defined in a neighborhood of $t=0$ such that $v_j(0)$ is an eigenvector of $A$ with eigenvalue $\nu$.  Since the eigenvectors $x_n$ correspond to eigenvalues $\lambda_n$ that converge to $0$, $x_n$ will be orthogonal to each of the $v_j(t_n)$ when $n$ is large enough. Therefore $\lim_{n \rightarrow \infty} \inner{y_n, v_j(t_0)} = \lim_{n \rightarrow \infty} \inner{x_n, v_j(t_n)} = 0$.  Since $A_-$ is compact and self-adjoint, this implies that $A_- y_n \rightarrow 0$. We also know that $\inner{A y_n, y_n} \rightarrow 0$, so we must also have $\inner{A_+ y_n, y_n} = \| A_+^{1/2} y_n \|^2 \rightarrow 0$. Since both $A_-$ and $A_+^{1/2}$ are positive operators, both have graphs that are weakly closed. Then since $y_n \xrightarrow{w} x$, we conclude that $A_- x = 0$ and $A_+^{1/2} x = 0$.  This implies that $x \in \D(A)$ and $Ax =0$, so $0$ is an eigenvalue of $A$.
  
Since $P$ is the orthogonal projection onto the kernel of $A$, we have $Px = x$.  For any $y \in \HH$ we calculate 
\begin{align*}
\inner{x,y} &= \inner{Px,y} = \inner{x, Py} = \lim_{n \rightarrow \infty} \inner{x_n, Py} \\
&=  \lim_{n \rightarrow \infty} \inner{A(t_n)x_n, Py}/\lambda_n \\
&=  \lim_{n \rightarrow \infty} \inner{A(t_n)x_n, Py}/\lambda_n 
\end{align*}
Note that $\inner{A(t_n) x_n, Py} = \mf{a}(t_n)[x_n, Py]$.  Since $\mf{a}_0[x_n] \rightarrow 0$, \eqref{eq:M2} implies that there is a constant $C > 0$ such that  
$$|\mf{a}(t_n)[x_n, Py] - (\mf{a}_0+t_n \mf{a}_1)[x_n,Py]| \le C t_n^2$$
for all $n$. Therefore, 
\begin{align*}
\inner{x,y}&=  \lim_{n \rightarrow \infty} \inner{A(t_n)x_n, Py}/\lambda_n  \\
&= \lim_{n \rightarrow \infty} (\mf{a}_0+t_n \mf{a}_1)[x_n,Py]/\lambda_n \\
&= \lim_{n \rightarrow \infty} (\mf{a}_0[x_n,Py]+t_n \mf{a}_1[x_n,Py])/\lambda_n.
\end{align*}
The sesquilinear form $\mf{a}_0 + t \mf{a}_1$ is closed when $t$ is small by \cite[Theorem VI.1.33]{Kato}.
By Lemma \ref{lem:technical}, $\lim_{n \rightarrow \infty} (\mf{a}_0+t \mf{a}_1)[x_n, Py] = (\mf{a}_0+t\mf{a}_1)[x,Py]$ for all $0 \le t < \epsilon$.  Therefore $\lim_{n \rightarrow \infty} \mf{a}_1[x_n,Py] = \mf{a}_1[x,Py]$. Since $\D(A)$ is a core for $\mf{a}_0$ and $\mf{a}_0[v,Py] = \inner{Av,Py} = \inner{PAv,y} = 0$ for all $v \in \D(A)$, we can use a limiting argument to show that $\mf{a}_0[x_n,Py] = 0$ for all $n \in \N$. Then, continuing the expansion of $\inner{x,y}$ from above, we have, 
\begin{align*}
\inner{x,y} &= \lim_{n \rightarrow \infty} (\mf{a}_0[x_n,Py] + t_n\mf{a}_1[x_n,Py])/\lambda_n \\
&= \lim_{n \rightarrow \infty} (t_n/\lambda_n) \mf{a}_1[x_n,Py] \\
&= \beta^{-1} \mf{a}_1[x,Py] = \beta^{-1} \mf{a}_1[Px,Py] = \beta^{-1} \inner{B_0 x,y}.
\end{align*}

This implies that $\inner{B_0x,y} = 0$ for all $y \in x^\perp$, so $x$ is an eigenvector of $B_0$.  It also implies that $\inner{B_0x,x} = \beta \|x\|^2$, so $x$ has eigenvalue $\beta$.  

Observe that the limit points of $\lambda(t)/t$ as $t \rightarrow 0^+$ form an interval. Any other limit point $\gamma$ of $\lambda(t)/t$ that is sufficiently close to $\beta$ so that $i \gamma$ is not in $\We(\mf{a}_0+i\mf{a}_1)$ would also have to be an eigenvalue of $B_0$ by the argument above.  However, since $i\beta \notin \We(\mf{a}_0+i\mf{a}_1)$, $\beta$ must be an isolated eigenvalue of $B_0$ with finite multiplicity.  This implies that $\lim_{t \rightarrow 0^+} \lambda(t)/t = \beta$. 
\end{proof}

\begin{remark}
In many examples, there is a series expansion for the self-adjoint holomorphic family in Theorem \ref{thm:weakerMain}, that is,
$$A(t) = A + t A_1 + t^2 A_2 + \ldots$$
which converges for all $x \in \D(A)$ when $|t|$ is sufficiently small.  This is true for bounded families and also when $A(t)$ is type (A) or type (B$_0$) (see \cite{Kato}). In these cases $B_0 = P A_1 P$.   
\end{remark}

\begin{corollary} \label{cor:extendedWeakerMain}
Let $A(t)$ be a self-adjoint holomorphic family of type (B). 
Let $\Sigma(t) = \min \sige(A(t))$. If $\lambda(t) < \Sigma(t)$ is an eigenvalue of $A(t)$ that depends continuously on $t$ when $t_0 < t < \epsilon$, $\lim_{t \rightarrow t_0^+} \lambda(t) = \Sigma(t_0)$, and $\liminf_{t \rightarrow 0^+} (\lambda(t) - \Sigma(t))/(t-t_0) < 0$, then $\Sigma(t_0)$ is an eigenvalue of $A(t_0)$.  Furthermore, $\lim_{t \rightarrow t_0^+} (\lambda(t)-\Sigma(t))/(t-t_0)$ exists and  is equal to an eigenvalue of the self-adjoint operator $B_0$ defined by \eqref{eq:B0} when $P$ is the orthogonal projection onto the kernel of $A(t_0) - \Sigma(t_0)$.   
\end{corollary}

\begin{proof}
Apply Theorem \ref{thm:weakerMain} to the type (B) self-adjoint family $\tilde{A}(t-t_0) = A(t) - \Sigma(t_0)$. 
\end{proof}

If the operator $A=A(0)$ in Theorem \ref{thm:weakerMain} has only a finite number of negative eigenvalues, then we can say more about how the minimal eigenvalues of $A(t)$ behave as $t$ approaches zero from above.

\begin{theorem} \label{thm:main}
With the same conditions and notation as Theorem \ref{thm:weakerMain}, suppose in addition that $A=A(0)$ has exactly $m < \infty$ negative eigenvalues counting multiplicity. Let $\omega = \lim_{t \rightarrow 0^+} \Sigma(t)/t$, let $P$ denote the orthogonal projection onto the kernel of $A$, and let $B_0$ be defined as in \eqref{eq:B0}. Suppose that the smallest $k+1$ eigenvalues of $B_0$ (counting multiplicity) are all less that $\omega$. Label these eigenvalues
$$\mu_0 \le \mu_1 \le \ldots \le \mu_k,$$
and label the smallest $m+k+1$ eigenvalues of $A(t)$ as 
$$\lambda_{-m}(t) \le \ldots \le \lambda_0(t) \le \lambda_1(t) \le ...\le \lambda_k(t).$$
Then $\lim_{t \rightarrow 0^+} \lambda_k(t)/t = \mu_k$. Moreover, if $x_k(t)$ is a family of unit eigenvectors of $A(t)$ corresponding to $\lambda_k(t)$ for $0 < t < \epsilon$, then for any sequence $t_n \rightarrow 0^+$, $x_k(t_n)$ has a limit point that is a unit eigenvector of $B_0$ corresponding to $\mu_k$.  
\end{theorem}

\begin{proof}
Let $\mf{a}(t)$ be the type (a) family of sesquilinear forms corresponding to $A(t)$. Then $\mf{a}(t)$ has a series expansion \eqref{eq:powerSeries} by Lemma \ref{lem:powerSeries}.  Fix $b > 0$ large enough so that $-b < \min \sigma(A)$ and let $\|x\|_\mf{a} = \sqrt{\mf{a}_0[x]+\|x\|^2}$ for all $x \in \D(\mf{a})$.  
Let $v_j$ denote a unit eigenvector of $A$ corresponding to the eigenvalue $\lambda_j(0)$ for $-m \le j < 0$, and let $v_j$ denote a unit eigenvector of $B_0$ corresponding to $\mu_j$ for $0 \le j < k$.  Then $\lim_{t \rightarrow 0^+} \mf{a}(t)[v_j] = \mf{a}_0[v_j] = \lambda_j(0)$ when $-m \le j < 0$ and $(\mf{a}_0+t\mf{a}_1)[v_j] = t\mu_j$ for $0 \le j \le k$ and all $0 < t < \epsilon$. 
By \eqref{eq:M2} there is a constant $M_2 > 0$ such that,  $\mf{a}(t)[v_j] \le t \mu_j + M_2 t^2 \|v_j\|_\mf{a} = \mu_j t + M_2 b t^2$. Therefore the Courant-Fischer-Weyl min-max principal implies that $\lambda_k(t) \le t \mu_k + M_2 b t^2$ for all $t > 0$ sufficiently small.  We also know that $\lim_{t \rightarrow 0^+} \lambda_k(t)/t$ converges by Theorem \ref{thm:weakerMain}. Therefore $\lim_{t \rightarrow 0^+} \lambda_k(t)/t \le \mu_k$.  We just need to prove that $\lim_{t \rightarrow 0^+} \lambda_k(t)/t = \mu_k$.

Let $0 < t_0 < \epsilon$. The Kato-Rellich theory says that for any eigenvalue $\lambda < \Sigma(t_0)$ of $A(t_0)$, it is possible to find analytic functions $x(t)$ and $\lambda(t)$ defined for $t$ in a neighborhood of $t_0$ such that $x(t)$ is a unit eigenvector of $A(t)$, $\lambda(t)$ is the corresponding eigenvalue, and $\lambda(t_0) = \lambda$. As long as $\lambda(t) < \Sigma(t)$, the functions $x(t)$ and $\lambda(t)$ can be analytically continued. 
Observe that $\lambda(t) = \inner{A(t) x(t), x(t)} = \mf{a}(t)[x(t)]$.  We also know that $\lambda'(t) = \mf{a}_1[x(t)]$ for all $t$ \cite[Section VII.4.6]{Kato}.
By \eqref{eq:M2}
$$|\mf{a}(t)[x(t)]-(\mf{a}_0+t\mf{a}_1)[x(t)]| < M_2 t^2 \|x(t)\|_\mf{a}^2.$$
Therefore 
$$|\lambda(t)-t \lambda'(t) - \mf{a}_0[x(t)]| < M_2 t^2 \|x(t)\|_\mf{a}^2.$$   
This implies that
\begin{equation} \label{eq:rateBound}
\frac{d}{dt} \frac{\lambda(t)}{t} = \frac{t\lambda'(t)-\lambda(t)}{t^2} \le -\frac{\mf{a}_0[x(t)]}{t^2} + M_2 \|x(t)\|_\mf{a}^2.
\end{equation}

\color{black}
For all $x \in \D(A)$, we have $\mf{a}_0[x] = \inner{Ax,x} = \inner{A_+x, x} - \inner{A_-x,x}$ where $A_\pm$ are the positive and negative parts of the self-adjoint operator $A$, that is,
\begin{equation} \label{eq:positiveNegative}
A_+ = \tfrac{1}{2}(|A|+A) \text{ and } A_- = \tfrac{1}{2}(|A|-A).
\end{equation}  
Since $A$ has only a finite number of negative eigenvalues, $A_-$ is a compact operator.  There is also an analytic projection operator $P(t)$ such that $P(0)$ is the spectral projection corresponding to the negative eigenvalues of $A$, and $P(t)$ is the spectral projection corresponding to the eigenvalues of $A(t)$ that are in a neighborhood of the negative eigenvalues of $A(0)$ when $t > 0$ is sufficiently small. Let $Q(t) = I - P(t)$. Then $Q(t)$ has a power series expansion
$$Q(t) = Q_0 + t Q_1 + t^2 Q_2 + \ldots$$
that is absolutely convergent in norm when $|t| < r$ for some $r > 0$. This follows from the fact that $Q(t)$ is bounded-holomorphic using Cauchy's Inequality, see \cite[Chapter VII]{Kato} for details.  

The series expansion for $Q(t) A_- Q(t)$ is
$$Q(t) A_- Q(t) = Q_0 A_- Q_0 + (Q_0 A_- Q_1  + Q_1 A_- Q_0)t + o(t^2).$$
This series is also absolutely converging in norm, and since $Q_0 A_- = A_- Q_0=0$, we see that 
$\|Q(t) A_- Q(t)\|/t^2$ is bounded by some constant $M_0 > 0$ in a neighborhood of $t =0$.  

Let $\mf{a}_-$ be the sesquilinear form corresponding to $A_-$, and define $\mf{a}_+$ to be $\mf{a}_0 + \mf{a}_-$.  Then $\D(\mf{a}_+) = \D(\mf{a})$ since $\D(\mf{a}_-) = \HH$.  Also, $\D(A)$ is a core for $\mf{a}_0$, so it is also a core for $\mf{a}_+$.  This implies that $\mf{a}_+[x] \ge 0$ for all $x \in \D(\mf{a})$ since $A_+ \ge 0$.  

Suppose that $x \in \D(A(t))$ has $\|x\| = 1$, and $Q(t)x=x$.  Then 
$$
\mf{a}_0[x] = \mf{a}_+[x] - \mf{a}_-[x] \ge -\inner{A_- x, x} = - \inner{Q(t)A_-Q(t)x,x} \ge - M_0t^2.
$$
Combined with \eqref{eq:rateBound} and the fact that $\mf{a}_0[x(t)]$ is bounded, this implies that for any eigenvalue $\lambda(t) < \Sigma(t)$ of $A(t)$ corresponding to an eigenvector $x(t)$ such that $Q(t)x(t) = x(t)$, there is a constant $M > 0$ such that 
\begin{equation} \label{eq:lowerbound}
\frac{d}{dt} \frac{\lambda(t)}{t} \le M \text{ and } \lambda'(t) \le \frac{\lambda(t)}{t} + M t
\end{equation}
when $t > 0$ is sufficiently small.

\color{black}
Recall that $\lambda'_k(t)$ is defined and analytic, except at isolated points where the analytic curves corresponding to the eigenvalues of $A(t)$ near $\lambda_k(t)$ cross. Then \eqref{eq:lowerbound} implies that $\lambda_k'(t) \le \mu_k + Mt$ wherever it is defined. In particular, if $\delta > 0$ is small enough so that $\mu_k + \delta < \omega$, then for all $t< \delta/M$, we have $\lambda_k'(t) \le \mu_k + \delta < \omega$.  Even at points where $\lambda_k'(t)$ is not defined, each of the analytic eigenvalue curves that cross $\lambda_k(t)$ at that point will have a derivative at most $\mu_k + \delta$. The same argument applies to each $\lambda_j(t)$ for $0 \le j \le k$. For each $0 < t < \delta/M$, we can choose a mutually orthogonal collection of unit eigenvectors $x_j(t)$ of $A(t)$ corresponding to $\lambda_j(t)$ for $0 \le j \le k$. For any sequence $t_n \rightarrow 0^+$, we can take a subsequence such that each $x_j(t_n)$ converges weakly to some $x_j$ for each $0 \le j \le k$. Then the argument of Theorem \ref{thm:weakerMain} implies that each $x_j$ is a nonzero eigenvector of $B_0$ with an eigenvalue equal to $\lim_{t \rightarrow 0^+} \lambda_j(t)/t$. In addition, the eigenvectors $x_j$ are mutually orthogonal. If $\lim_{t \rightarrow 0^+} \lambda_k(t)/t < \mu_k$, then $B_0$ has $k+1$ mutually orthogonal eigenvectors with eigenvalues strictly less than $\mu_k$, but that contradicts the Courant-Fischer-Weyl min-max principal. Therefore we conclude that $\lim_{t \rightarrow 0^+} \lambda_k(t)/t = \mu_k$.  

Now consider any sequence $t_n \rightarrow 0^+$ such that $x_k(t_n)$ converges weakly to $x_k$.  We know that $x_k$ is an eigenvector of $B_0$ with eigenvalue equal to $\mu_k$.   Now consider 
$$\mf{a}_1[x_k(t_n) -x_k] = \mf{a}_1[x_k(t_n)] - \mf{a}_1[x_k(t_n),x_k] - \mf{a}_1[x_k,x_k(t_n)] + \mf{a}_1[x_k].$$
It was observed in the proof of Theorem \ref{thm:weakerMain} that $\lim_{n \rightarrow \infty} \mf{a}_1[x_k(t_n), x] = \mf{a}_1[x_k] = \mu_k \|x_k\|^2$.  Therefore 
$$\lim_{n \rightarrow \infty} \mf{a}_1[x_k(t_n)] = \lim_{n \rightarrow \infty} \mf{a}_1[x_k(t_n)-x_k] + \mu_k \|x_k\|^2. $$
By construction, $x_k(t_n) - x_k \xrightarrow{w} 0$ and a quick calculation shows that 
$$\lim_{n \rightarrow \infty} \|x_k(t_n)-x_k\|^2 = 1 - \|x_k\|^2.$$  
We also know that $\mf{a}_1[x_k(t_n)-x_k]$ is bounded by \eqref{eq:Bbound} since $\mf{a}_0[x_k(t_n)-x_k] \rightarrow 0$.  Therefore we can pass to a subsequence such that $\mf{a}_1[x_k(t_n)-x_k]/\|x_k(t_n)-x_k\|^2$ converges to some $z \in \We(\mf{a}_1)$.  By Proposition \ref{prop:bottom}, $z \ge \omega$.  Then, we have 
$$\lim_{n \rightarrow \infty} \mf{a}_1[x_k(t_n)] = (1-\|x_k\|^2) z + \mu_k \|x_k\|^2 \ge \mu_k.$$
Recall from the Kato-Rellich theory that $\lambda_k'(t) = \mf{a}_1[x_k(t)]$ wherever the derivative is defined. Even where the derivative is not defined because multiple analytic eigenvalue curves cross at $\lambda_k(t)$, each of the crossing curves will have a derivative at most $\lambda_k(t)/t + M t$ by \eqref{eq:lowerbound} and therefore $\lim_{n \rightarrow \infty} \mf{a}_1[x_k(t_n)] = \mu_k$.  Then $\|x_k\| = 1$, so $x_k(t_n)$ converges to $x_k$ in norm.  
\end{proof}

If $B$ is a self-adjoint operator on $\HH$ that is not bounded below, it is still possible to define a sesquilinear form corresponding to $B$. We use the spectral theorem to decompose $B$ into positive and negative parts, that is, $B = B_+ - B_-$ where $B_\pm$ are given by \eqref{eq:positiveNegative}. 
Then define the sesquilinear form $\mathfrak{b}$ associated with $B$ to be 
$$\mathfrak{b}[x,y] = \inner{B_+^{1/2}x,B_+^{1/2}y} - \inner{B_-^{1/2}x, B_-^{1/2}y},$$ 
and $\D(\mathfrak{b}) =\D(|B|^{1/2})$.  Theorem \ref{thm:Simon} is a special case of Corollary \ref{cor:extendedWeakerMain} and Theorem \ref{thm:main} because the assumption that $B$ is relative $A$-form compact implies that the essential numerical range of the corresponding sesquilinear form $\mf{a}+i \mf{b}$ is contained in the ray $[0,\infty)$. 
\begin{lemma} \label{lem:relcompact}
Let $A, B$ be self-adjoint operators on a Hilbert space $\HH$ with corresponding quadratic forms $\mathfrak{a}$ and $\mathfrak{b}$, respectively. Let $\mathfrak{t} = \mathfrak{a}+i\mathfrak{b}$. Suppose that $A \ge 0$ and $B$ is relative $A$-form compact, that is, $|B|^{1/2}(A+I)^{-1}|B|^{1/2}$ is compact. Then $\We(\mathfrak{t}) \subset [0,\infty)$ and if $A(t)$ is the type (B) family of self-adjoint operators corresponding to the sesquilinear form $\mf{a}+t\mf{b}$, then $\min \sige(A(t)) = 0$ for all $t \in \R$.  
\end{lemma}
\begin{proof}
Consider any $z \in \We(\mathfrak{t})$.  There is a sequence $x_n$ in $\D(\mathfrak{t})$ with $\|x_n\| = 1$, $x_n \xrightarrow{w} 0$, such that $\mathfrak{t}[x_n]$ converges to $z$.  In particular $\mathfrak{a}[x_n]$ is bounded.  Note that $\D(\mf{a}) = \D(A^{1/2}) = \D((A+I)^{1/2})$ (see e.g., \cite[Problem VI.2.25]{Kato}) so
$$\inner{(A+I)^{1/2} x_n, (A+I)^{1/2} x_n} =  \mathfrak{a}[x_n] + \|x_n\|^2$$
for all $n$, and therefore $\|(A+I)^{1/2} x_n\|$ is bounded. By passing to a subsequence, we can assume that $y_n = (A+I)^{1/2} x_n$ converges weakly to some $y \in \HH$. 

Since $x_n = (A+I)^{-1/2} y_n$ and the graph of $(A+I)^{-1/2}$ is weakly closed, it follows that $x_n \xrightarrow{w} (A+I)^{-1/2} y$. We also know that $x_n \xrightarrow{w} 0$, so we conclude that $(A+I)^{-1/2} y = 0$. 

For each $n$, let $\tilde{y}_n = y_n - y$.  Then $\tilde{y}_n \xrightarrow{w} 0$.  Observe that 
\begin{align*}
|B|^{1/2} x_n &= |B|^{1/2}(A+I)^{-1/2} y_n \\
&= |B|^{1/2}(A+I)^{-1/2} \tilde{y}_n + |B|^{1/2}(A+I)^{-1/2} y \\
&= |B|^{1/2}(A+I)^{-1/2} \tilde{y}_n
\end{align*}
This converges to zero because $|B|^{1/2}(A+I)^{-1/2}$ is compact. Therefore $|B|^{1/2} x_n \rightarrow 0$ and so $\mathfrak{b}[x_n] \rightarrow 0$ which proves that $\We(\mathfrak{t}) \subset \R$. Since $A \ge 0$, it follows that $\We(\mathfrak{t}) \subset [0,\infty)$. Then we use Lemma \ref{lem:sectorial} to observe that
\begin{align*}
\min \sige(A(t)) = \min \We(A(t)) &= \min \We(\mf{a}+t \mf{b}) \\
&= \min \re \We((1-it)\mf{t}) = 0 
\end{align*} 
for all $t \in \R$.  
\end{proof}

\begin{remark}
Lemma \ref{lem:relcompact} and Corollary \ref{cor:extendedWeakerMain} together imply that if 
$$\lim_{t \rightarrow t_0^+} \mu(t)/(t-t_0) < 0,$$ 
in the notation of Theorem \ref{thm:Simon}, then $0$ is an eigenvalue of $A \dotplus t_0B$. If zero is not an eigenvalue of $A$, then the converse is also true, as was observed in \cite{Simon77}.  Let $A(\tau)$ denote the type (B) family of self-adjoint operators corresponding to the family of sesquilinear forms $(\mf{a}+t_0 \mf{b}) + \tau \mf{b}$ (here $\tau = t - t_0$).   Suppose that 0 is an eigenvalue of $A(0)$ (here $A(0) = A \dotplus t_0B$), and $x$ is an eigenvector of $A(0)$ corresponding to $0$.  Then $(\mf{a}+t_0 \mf{b})[x] = 0$.  Since $A \ge 0$, $\mf{b}[x] < 0$.  Therefore the self-adjoint operator $B_0$ defined by \eqref{eq:B0} has a minimal eigenvalue $\beta < 0$. If $\mu(\tau)$ is the sole negative eigenvalue of $A(\tau)$ that approaches 0 as $\tau \rightarrow 0^+$, then $\mu(\tau)/\tau \rightarrow \beta$ by Theorem \ref{thm:main}.
\color{black}
\end{remark}

\section{Isolated Eigenvalues} \label{sec:isolated}
If $A(t)$ is a family of self-adjoint operators that depend analytically on the real parameter $t$, then the Kato-Rellich perturbation theory applies to the isolated eigenvalues of $A(t)$ with finite multiplicity. In general, it is not possible to analytically continue an eigenvalue function $\lambda(t)$ after it approaches an element of the essential spectrum, see Example \ref{ex:Narco}. In some circumstances, the Kato-Rellich theory can be adapted to isolated eigenvalues with infinite multiplicity, as the following theorem shows. 

\begin{theorem} \label{thm:isolated}
Suppose that $A(t)$ is a holomorphic family of self-adjoint bounded linear operators on $\mathcal{H}$ with power series expansion 
$$A(t) = A_0 + t A_1 + t^2 A_2 + \ldots$$
defined in a neighborhood of $t = 0$.  Suppose that $0$ is an isolated element of the spectrum of $A_0$.  Let $P$ denote the spectral projection onto the kernel of $A_0$.  If $\mu$ is an element of the discrete spectrum of $P A_1 P$ with multiplicity $k$, then there is a family of $k$ analytic functions $x_j(t) \in \HH$ defined on an open interval $I$ containing $0$ such that the collection $x_j(t)$, $j = 1, \ldots, k$, is a mutually orthogonal family of unit eigenvectors of $A(t)$ with corresponding eigenvalues $\lambda_j(t)$ that satisfy $\lambda_j(0) = 0$ and $\lambda_j'(0) = \mu$ for all $1 \le j \le k$.
\end{theorem}			
\begin{proof}
Since $0$ is an isolated element of the spectrum of $A(0)$, we can construct an analytic spectral projection function $P(t)$ such that $P(0) = P$ \cite[Theorem VII.1.7]{Kato}.  We will show that the expression $B(t) = t^{-1} P(t) A(t) P(t)$ has a power series expansion that converges in a neighborhood of $t=0$.  The spectral projection $P(t)$ has a power series of the form
$P(t) = P_0 + P_1 t + P_2 t^2 + \ldots$. Expanding the power series for $P(t) A(t) P(t)$ gives:
$$P_0 A_0 P_0+ t(P_0 A_1 P_0 + P_1 A_0 P_0 + P_0 A_0 P_1) + o(t^2)$$
Observe that $P_0 A_0 = A_0 P_0 = 0$, so the expression above simplifies to:
$$t (P_0 A_1 P_0) + o(t^2).$$
This means that $B(t) = t^{-1} P(t) A(t) P(t)$ is analytic in a neighborhood of $t = 0$.  Also, $\mu$ is an isolated eigenvalue of $B(0)$ with multiplicity $k$.  Therefore the Kato-Rellich perturbation theory applies, so there is a family of mutually orthogonal unit eigenvectors $x_1(t), \ldots, x_k(t)$ of $B(t)$ that are analytic functions of $t$ in an interval of $0$, and such that $B(0) x_j(0) = \mu x_j(0)$ for all $j$.  Each $x_j(t)$ is an eigenvector of $A(t)$ with corresponding eigenvalue $\lambda_j(t) = \inner{A(t)x_j(t), x_j(t)} = t \inner{B(t) x_j(t), x_j(t)}$.  Let $\mu_j(t)$ denote $\inner{B(t) x_j(t), x_j(t)}$ and observe that $\mu_j(t)$ is analytic in $t$ for all $1 \le j \le k$.  Then $\lambda_j'(t) = \mu_j(t) + t \mu_j'(t)$ and $\lambda_j'(0) = \mu_j(0) = \mu$ for all $1 \le j \le k$.
\end{proof}

\section{Examples} \label{sec:ex}


\begin{example} \label{ex:Volterra}
Let $\mathcal{H} = L^2(0,1)$.  The Volterra operator $V: \mathcal{H} \rightarrow \mathcal{H}$ is 
$$(Vf)(t) := \int_0^t f(s) \, ds.$$
It is well known that the Volterra operator is a compact linear operator.
The adjoint of $V$ is $(V^*f)(t) = \int_t^1 f(s) \, ds$ and therefore the real part of $V$ is a rank one self-adjoint operator with non-zero eigenvalue equal to 1/2 and the corresponding eigenspace consists of all constant functions. Let $V(\theta)$ denote the real part of $e^{-i \theta} V$ and note that  
$$V(\theta) = \tfrac{1}{2} ( e^{-i \theta} V + e^{i \theta} V^* ) = \cos \theta \re V + \sin \theta \im V.$$
Suppose that $f$ is a unit eigenvector of $V(\theta)$ for $\theta \in \R \backslash \{0\}$.  Then $V(\theta) f = \lambda f$ for some $\lambda \in \R$ and 
$$\lambda f'(x) = \tfrac{1}{2} e^{-i \theta} f(x) - \tfrac{1}{2} e^{i \theta} f(x) = -i\sin \theta f(x).$$ 
This means that the eigenvectors of $V(\theta)$ have the form $f(x) = e^{-i x \sin \theta/\lambda}$ when $\lambda \ne 0$.  By substituting into the expression for $V(\theta)$ we see that the eigenvalues and corresponding unit eigenvectors of $V(\theta)$ are 
$$\lambda_n = \frac{\sin \theta}{ 2\theta + 2n \pi}, ~ f_n(t) = e^{-i t (2 \theta + 2n \pi)} \text{ where } n \in \Z.$$
In particular, $\lambda = 0$ is an isolated eigenvalue of $V(0)$ with infinite multiplicity. The power series expansion for $V(\theta)$ centered at $\theta = 0$ is
$$V(\theta) = \re{V} + \theta \im{V} + o(\theta^2).$$
As predicted by Theorem \ref{thm:isolated}, the eigenvectors of $V(\theta)$ are analytic functions of $\theta$ and can be continued analytically even when $\theta = 0$. The corresponding eigenvalues at $\theta = 0$ are the eigenvalues of $P (\im V) P$ where $P$ is the orthogonal projection onto the kernel of $\re V$. Theorems \ref{thm:weakerMain} and \ref{thm:main} also apply to $V(\theta)$, but they do not show that the eigenvalues and eigenvectors of $V(\theta)$ can be analytically continued through $\theta = 0$.

\end{example}

\begin{example} \label{ex:Narco}
Let $H$ be the operator on $\ell_2(\N)$ defined by 
$$H(x)_k = \begin{cases}
0 & \text{ if } k = 1.  \\
e^{-k} x_k & \text{ otherwise.}  
\end{cases}
$$
We also choose two elements $a, b \in \ell_2(\N)$ with nonnegative entries such that 
$$a_k^2 = \begin{cases}
1 & \text{ if } k = 1, \\
\frac{3}{4^n} & \text{ if } k = (4n)^2 \text{ for } n \in \N, \\
0 & \text{otherwise,}
\end{cases}
$$
and 
$$b_k^2 = \begin{cases}
1 & \text{ if } k = 1, \\
\frac{1}{2} & \text{ if } k = 2, \\
\frac{3}{2 (4^n)} & \text{ if } k = (4n+2)^2 \text{ for } n \in \N, \\
0 & \text{otherwise.}
\end{cases}
$$
Observe that $\|a_k \| = \|b_k \| = \sqrt{2}$. In particular, $\sum_{k = 2}^\infty a_k^2 = \sum_{k  = 2}^\infty b_k^2 = 1$.  We also note that 
$$\sum_{k = 2}^{(4n+1)^2} a_k^2 = \sum_{k = 2}^{(4n+3)^2} a_k^2 = 1 - \frac{1}{4^{n}},$$
while 
$$\sum_{k = 2}^{(4n+1)^2} b_k^2 = 1 - \frac{2}{4^{n}} \text{ and } \sum_{k = 2}^{(4n+3)^2} b_k^2 =  1 - \frac{1}{2(4^{n})}. $$


Let $K_a$ be the rank one operator $K_a(x) = \inner{x,a} a$ and likewise let $K_b(x) = \inner{x,b}b$.  Since $H \ge 0$, and $\inner{H-tK_a e_1,e_1} = -t$ where $e_1 = (1,0,0,\ldots) \in \ell_2(\N)$, it follows from that $H-tK_a$ must have at least one negative eigenvalue.  However, since $K_a$ is rank one, the Courant-Fischer-Weyl max-min principal implies that $H-tK_a$ has only one negative eigenvalue, counting multiplicity.  The same argument applies to $H-tK_b$. Let $\lambda_a(t)$ denote the minimal eigenvalue of $H-t K_a$ and let $\lambda_b(t)$ denote the minimal eigenvalue of $H-t K_b$.
Both $\lambda_a(t)$ and $\lambda_b(t)$ are analytic functions of $t$ when $0 < t < \infty$ by the Kato-Rellich theory. We will show that the minimal eigenvalues of $H-t K_a$ and $H-t K_b$ cross infinitely many times as $t \rightarrow 0^+$, and therefore at least one of the two eigenvalue functions $\lambda_a(t)$ or $\lambda_b(t)$ cannot be analytically continued beyond $t=0$.

Fix $0< t< 1$ and suppose that 
$(H - tK_a) x = \lambda x$ for $x \in \ell_2$ with $\|x\|=1$ and $\lambda < 0$.  Then 
\begin{align*} 
H x - t K_a x &= \lambda x \\
( H - \lambda I) x &=  t \inner{x, a} a.
\end{align*}
Since $H \ge 0$ and $\lambda < 0$, $H - \lambda I$ is invertible, so 
$$x = t \inner{x, a}  (H - \lambda I)^{-1} a$$
Taking the inner-product of both sides of the expression above with $a$ and solving for $t$, we get 
$$\frac{1}{\inner{(H - \lambda I)^{-1} a,a}} = t $$ 
The expression $\inner{(H - \lambda I)^{-1} a,a}$ can be expanded as:
$$ \frac{1}{-\lambda} a_1^2 + \sum_{k = 2}^{\infty} \frac{1}{e^{-k} - \lambda} a_k^2.$$
It is apparent that this expression is a strictly monotone function of $\lambda \in (-\infty,0)$ and therefore so is $t$.  For convenience, let $f_a(\lambda) = \inner{(H - \lambda I)^{-1} a,a}$, and likewise, let $f_b(\lambda) = \inner{(H - \lambda I)^{-1} b,b}$.  The minimum eigenvalues of $H - t K_a$ and $H - tK_b$ are given by $\lambda_a(t) = f_a^{-1}(1/t)$ and $\lambda_b(t) = f_b^{-1}(1/t)$ respectively. We will show that the functions $f_a(\lambda)$ and $f_b(\lambda)$ cross infinitely many times as $\lambda$ approaches $0$ from below.  This in turn will show that the functions $\lambda_a(t)$ and $\lambda_b(t)$ cross infinitely many times as $t \rightarrow 0^+$.  

If $k \le (m-1)^2$, then  
$$\frac{1}{1+e^{m^2 - k}} \le  \frac{1}{1+e^{2m - 1}} < \frac{1}{2^m},  $$
and similarly if $k \ge (m+1)^2$, then 
$$ \frac{1}{1+e^{m^2 - k}} \ge   \frac{1}{1+e^{-2m - 1}}  > \frac{2^m-1}{2^m}.$$
If $m = 4n+1$ and $\lambda = -e^{-m^2}$, then 
\begin{align*}
f_a(\lambda) &= e^{m^2} \left(1 + \sum_{k=2}^\infty \left( \frac{1}{1+e^{m^2-k}} \right) a_k^2 \right) \\
&<  e^{m^2} \left(1 + \sum_{k = 2}^{m^2}  \left( \frac{1}{1+e^{m^2-k}} \right) a_k^2 + \sum_{k = m^2}^{\infty} a_k^2 \right)  \\
&<  e^{m^2} \left(1 + \sum_{k = 2}^{m^2} \left( \frac{1}{2^{m}} \right) a_k^2 + \sum_{k = m^2}^{\infty} a_k^2 \right)  \\
&= e^{m^2} \left(1 + \left( \frac{1}{2^{m}} \right)  \left(1-\frac{1}{4^n} \right) + \frac{1}{4^n} \right) \\
&= e^{m^2} \left(1 + \frac{1}{2^m} + \left( 1 - \frac{1}{2^m} \right) \frac{1}{4^n} \right),
\end{align*}


while
\begin{align*}
f_b(\lambda) &=  e^{m^2} \left(1 + \sum_{k = 2}^{\infty} \left( \frac{1}{1+e^{m^2-k}} \right) b_k^2 \right)   \\
 &>  e^{m^2} \left(1 + \sum_{k = m^2}^{\infty} \left( \frac{1}{1+e^{m^2-k}} \right) b_k^2 \right)   \\
&>  e^{m^2} \left(1 + \sum_{k = m^2}^{\infty} \left( \frac{2^m-1}{2^m} \right) b_k^2  \right)  \\
&= e^{m^2} \left(1 + \left( \frac{2^m-1}{2^m} \right)  \left(\frac{2}{4^n} \right) \right) \\
&= e^{m^2} \left(1 + \left( 1 - \frac{1}{2^m} \right) \frac{2}{4^n} \right).
\end{align*}

By inspection, it is clear from the above inequalities that $f_a(\lambda) < f_b(\lambda)$ when $\lambda = -e^{-(4n+1)^2}$. Essentially the same argument shows that $f_b(\lambda) < f_a(\lambda)$ when $\lambda = -e^{-(4n+3)^2}$. Therefore the functions $f_a(\lambda)$ and $f_b(\lambda)$ cross infinitely many times as $\lambda \rightarrow 0^-$, which proves that at least one of the eigenvalue functions $\lambda_a(t)$ or $\lambda_b(t)$ cannot be analytically continued past beyond $t=0$.  
\end{example}

%
%

\bibliographystyle{plain}
\bibliography{perturb}
\end{document}